\documentclass[12pt]{amsart}
\usepackage{amssymb,amsmath}
\usepackage{amsthm}
\usepackage{indentfirst}
\usepackage{amscd}
\usepackage{graphicx}
\numberwithin{equation}{section}

\newtheorem{Theorem}{Theorem}[section]
\newtheorem{Definition}[Theorem]{Definition}

\newtheorem{Lemma}[Theorem]{Lemma}

\newtheorem{Assumption-Notation}[Theorem]{Assumption-Notation}
\newtheorem{Question}{Question}

\newtheorem{Remark}[Theorem]{Remark}

\newtheorem{Problem}{Problem}

\newtheorem{Example}[Theorem]{Example}

\def\triangleq{\stackrel{\triangle}{=}}

\def\e{{\varepsilon}}

\newcommand{\T}{{\mathbb{T}^2}}

\newcommand{\U}{ \mathcal{U}}

\renewcommand{\e} {\varepsilon}

\long\def\symbolfootnote[#1]#2{\begingroup
\def\thefootnote{\fnsymbol{footnote}}\footnote[#1]{#2}\endgroup}

\begin{document}

\title{non-integrability of dominated splitting on $\T$ }
\footnote{B. H is supported by CPSF 2013M540805. S. G is supported by 973 project 2011CB808002, NSFC 11025101 and 11231001.}

\address {Mathematics and Science College\\Shanghai Normal University\\Shanghai 200235\\P.R.China}
\address{Beijing International Center for Mathematical Research\\Peking University\\Beijing 100871\\P.R.China}
\email{   hebaolin@shnu.edu.cn}
\address{School of Mathematics Sciences\\Peking University\\Beijing 100871\\P.R.China}
\email{gansb@math.pku.edu.cn}

\author{Baolin He and Shaobo Gan}
\maketitle

\textbf{Abstract} { We construct a diffeomorphism $f$ on 2-torus with a dominated splitting $E \oplus F$ such that there exists an open neighborhood $\mathcal{U} \ni f$ satisfying that for any $g \in \U$,  neither $E_g$ nor $F_g$ is integrable.}

\section{ Introduction}
According to the theory of Ordinary Differential Equations, Lipschitz vector fields are uniquely integrable. However, the bundles appeared in dynamics are mostly H{\"{o}}lder \cite{PSW}. Due to the hyperbolicity, the stable and unstable bundles are uniquely integrable. Particularly for two-dimensional $C^2$ Anosov diffeomorphisms, the two hyperbolic bundles are $C^1$ \cite{AS}! But, we know little on the integrability of center bundles, which is a really challenging problem \cite{BBI}.

In this paper, we focus on 2-dimensional diffeomorphisms on 2-torus $\T$ with dominated splitting. At first, we recall some related definitions.

Let $E$ be a one-dimensional continuous sub-bundle of $ \mathrm{T}\T$.
\begin{Definition}
 $E$ is said to be integrable if there exists a 1-foliation (continuous partition consisting of immersed 1-dimensional sub-manifolds) of $\T$ tangent to $E$.
\end{Definition}
\begin{Definition}
$E$ is said to be uniquely integrable if there exists exact one 1-foliation of $\T$ tangent to $E$.
\end{Definition}
\begin{Definition}
A $Df$-invariant bundle splitting $E \oplus F = \mathrm{T}\T$ is said to be a dominated splitting, if for any $x \in \T$, any unitary $u \in E_x$ and any unitary $v \in F_x$, $|Df(u)| < |Df(v)|$.
\end{Definition}
Both the two bundles in the splitting are continuous and uniquely defined. And, the dominated splitting is robust: there exists neighborhood $\U \ni f$ such that for any $g \in \U$, $g$ has the dominated splitting $E_g \oplus F_g$.

According to Peano's Theorem, for a continuous vector field, through every point $x$ there exists an integral curve. But, can these curves form a foliation?
\begin{Question}
Let $f$ be a diffeomorphism on $\T$ with a dominated splitting $E \oplus F$. Are these two sub-bundles integrated to foliations? Moreover, if $f$ is $C^2$,  are the two bundles Lipschitz ($C^1$)?
\end{Question}
For partially hyperbolic systems (one of $E$ and $F$ is uniformly hyperbolic), Pujals and Sambarino firstly gave a positive answer for the former question. For the latter, it is still unclear.
\begin{Theorem}\cite{PS,Po}
For partially hyperbolic diffeomorphisms on 2-torus $\T$, the two bundles in the dominated splitting are uniquely integrable.
\end{Theorem}
In this paper, we give a negative answer for the question:
\begin{Theorem}\label{thm} There exists a diffeomorphism $f$ on 2-torus with a dominated splitting $E \oplus F$ such that, there is an open neighborhood $\U \ni f$ satisfying that for any $g \in \U$, neither $E_g$ nor $F_g$ is integrable and hence neither of them is Lipschitz.
\end{Theorem}
In our construction, the non-integrability happens in a small neighborhood of sink (source). On the contrary, in \cite{PS}, it has an interesting corollary that \textit{ ``for any $C^2$ diffeomorphism on 2-torus with dominated splitting, if periodic points are all hyperbolic saddles, then the two bundles are uniquely integrable''}. How about $C^1$ systems:
\begin{Problem}
Given a $C^1$ diffeomorphism on 2-torus with a dominated splitting, if periodic points are all hyperbolic saddles, are the two bundles integrable?
\end{Problem}
Between the integrability and unique integrability, there exists such surprising phenomena for a H$\ddot{o}$lder continuous vector field on the plane: \textit{there are uncountable distinct foliations tangent to the vector field }\cite{BF}. So, it is natural to ask:
\begin{Problem}
Is there such diffeomorphism with dominated splitting $E \oplus F$ satisfying that, $E$($F$) is integrated to different foliations?
\end{Problem}
\section{two basic lemmas }
At first, we introduce some notations used through the paper. Let
\[A={\begin{pmatrix}
 2 & 1 \\
 1&1
\end{pmatrix}}^2.\]
Let $0 < \lambda <1 < \mu$ be the two eigenvalues of $A$, $E^s$ be the contracting eigenspace of $A$, and $E^u$ be the expanding eigenspace of $A$.
Let $f_A$ be the hyperbolic automorphism on 2-torus induced by $A$, which has two fixed point at least.  $E^s$ and $E^u$ induces the hyperbolic splitting of $f_A$, still denoted as $E^s \oplus E^u = \mathrm{T}\T$; and the two eigenspaces induce the coordinate $\{ \frac{\partial} {\partial{x_1}}, \frac{\partial} {\partial{x_2}} \}$ on 2-torus.

Let $f$ be a diffeomorphism on $\T$, the norm of $Df$ is denoted by
$$||Df|| = \max \{|Df(v)|/|v|: 0\not= v \in \mathrm{T}\T \}.$$
The norm of $Df$ restricted on a sub-bundle $E$, is denoted by $||Df|_{E}||$.\\
For a hyperbolic fixed non-sink $x$ with the dominated splitting $T_x (M)= E(x) \oplus F(x)$, we define the strong unstable manifold $W^{uu}_{\frac{1}{10}}(x,f)$ as:
$$\{ y: d(f^{-n}(y),x)<\frac{1}{10},{\rm \ and \ }\exists N, {\rm\ s.t. \ }\frac{d(f^{-n}(y),x)}{|| Df^{-n}|_{E(x)}||} < \frac{1}{2} ,\ \forall n > N \}.$$
Similarly, we can define strong stable manifold $W^{ss}_{\frac{1}{10}}(x,f)$ for a hyperbolic fixed non-source $x$ with dominated splitting.\\

Now we give two basic lemmas. Firstly, we recall DA-operation\cite{Wi}:
\begin{Lemma}\label{da}
Let $p=(0,0)$ be a fixed point of $f_A$. Then, for any $\e >0$, there exists $C^0$-perturbation $f$ of $f_A$ such that:
\begin{enumerate}
\item[(1)] { $f(x)=f_A(x)$ outside the $\e$-ball $B(p,\e)$;}
\item[(2)] { for any $x \in M$,
\[Df(x)=
\begin{pmatrix} a(x) & b(x) \\
0 & \mu
\end{pmatrix} ,\]
here $\lambda^2<a(x) <\sqrt{\mu},\ |b(x)| < \e$ ;}
\item[(3)] {In $W^{s}(p,f_A)$, $f$ has exactly three periodic points contained in $B(p,\e)$: one fixed source $p$ and two fixed saddles;}
\item[(4)] {In some neighborhood of the two saddles above, $Df$ are constant diagonal matrixes;}
\item[(5)] {$W^{uu}_{\frac{1}{10}}(p,f) = \{0\} \times (-\frac{1}{10},\frac{1}{10}).$}
\end{enumerate}
Similarly, there exists a symmetrical DA-operation of $f_A$: to do the same DA-operation of $f_A^{-1}$.
\end{Lemma}

For completion, we give a proof of this basic lemma in the following.
\begin{proof}
Let $I_1 \times I_2 \subset B(p,\e)$, $I_1$ and $I_2$ both are intervals centered at 0, and
$$\ell{(I_1)} =\frac{\e}{3\mu} \ell(I_2).$$
Take a smooth bump function $\alpha$ satisfying the following conditions:
\begin{enumerate}
\item[(1)] {$\alpha(x)$ is an odd function and $\alpha(x)=0$, for $x \not\in I_1$ ;}
\item[(2)] {$\lambda^2-\lambda <\alpha' (x) < \sqrt{\mu}- \lambda$;}
\item[(3)] {$\alpha(x)+\lambda x$ has  exactly three periodic points all contained in $I_1$: one fixed source $0$ and two fixed sinks;}
\item[(4)] {In some neighborhood of the above two sinks, $\alpha'(x)$ are constant values.}
\end{enumerate}
Take another bump function $\beta$ satisfying that,
\[\begin{cases}
        \beta(x)=1, x \, \text{in a small neighborhood of} \, 0; \\
        \beta(x)=0, x \not\in I_2; \, 0\leq \beta(x) \leq 1;\\
        |\beta'(x)| < 3/{\ell(I_2)}
\end{cases}\]
Let
$$f(x_1,x_2)=(\alpha(x_1) \beta(x_2)+\lambda x_1 , \mu x_2).$$
Then,
\[Df=
\begin{pmatrix}\alpha'(x_1) \beta(x_2)+\lambda & \alpha(x_1) \beta'(x_2) \\
0 & \mu
\end{pmatrix}\]
Note that $\lambda^2- \lambda <\alpha' (x) < \sqrt{\mu}- \lambda$, $\beta \in [0,1]$ and $\ell{(I_1)} =\frac{\e}{3\mu} \ell(I_2).$\\
 Then,
$$|\alpha(x_1) \beta'(x_2)|<\mu \ell(I_1) \times \frac{3}{\ell(I_2)}=\e.$$
And,
$$\lambda^2<\alpha'(x_1) \beta(x_2)+\lambda  < \sqrt{\mu}.$$
This verifies the property (2) in the lemma.
From the properties (3) and (4) of function $\alpha$ and $\beta'(x)=0 $ in a small neighborhood of $0$, we get the properties (3) and (4) of the lemma. Since $Df$ are diagonal matrices on the line $\{0\} \times (-\frac{1}{10},\frac{1}{10})$, $f$ satisfies property (5).
\end{proof}
The next lemma is a classic theorem (e.g., see appendix B in \cite{BDV}), which gives a sufficient condition for a diffeomorphism to have a dominated splitting. For any two sub-bundles $E,F \subset \mathrm{T}\T$, $$\angle{(E,F)} \triangleq \max \{\angle{(u,v)}: \forall u\in E_x, v \in F_x, x \in \T \}$$
\begin{Lemma}\label{dmn}Let $\mathcal{K}>0$, $\eta >1, \delta>0$, there exists $\e>0$ such that for any diffeomorphism $f$ on $\T$  and
under the coordinate $\{ \frac{\partial} {\partial{x_1}}, \frac{\partial} {\partial{x_2}} \}$,
\[ Df(x)=\begin{pmatrix}
a(x) & b(x)\\
c(x) & d(x)
\end{pmatrix} \]
satisfying that for any $x \in \T$,
\[\begin{array}{ll}
\min \{ |a(x)|, |d(x)| \} > \mathcal{K},\\
 |d(x)| > \eta |a(x)|,\\
 \max \{ |b(x)|,|c(x)| \}  < \e,
\end{array}\]
$f$ has the dominated splitting $E\oplus F$ with the property that
$$\angle{(E,E^s)} < \delta, \quad \angle{(F,E^u)} < \delta.$$
\end{Lemma}

\section{A robustly non-integrable example}
Firstly, we construct a diffeomorphism on 2-torus with the dominated splitting $E \oplus F$ such that, $E$ is robustly non-integrable. It is a special DA-map: to do DA-operation twice.\\
\begin{Example}
 Let $p$ be a fixed point of $f_A$, $\e >0$ be a very small constant (to be determined in the following construction). By DA-operation in $B(p,\e)$ , we can take a map $g$ such that
\[ g(x) = f_A (x), x \not\in B(p,\e),\]
\[ Dg=\begin{pmatrix}
 a(x) & b(x)\\
0 & \mu
\end{pmatrix},\]
here, $\lambda^2 < a(x)< \sqrt{\mu}  ,|b(x)| < \e$. And, $g$ has the two fixed points: source $p$ and saddle $q \in B(p,\e)$.\\
Also, there exists a open neighborhood $U \ni q$ such that for any $x \in U$:
\[ Dg(x)=\begin{pmatrix}
a(q) & 0\\
0& \mu
\end{pmatrix}.\]
In this smaller neighborhood $U$, we make another DA-operation $f$ of $g$ such that,\\
$f$ has the two fixed points source $p$ and sink $q$ in $B(p,\e)$, both the length of two components of $W^{uu}_{ \frac{1}{10}}(p,f)-p$ equals $\frac{1}{10}$, and
\[Df=\begin{pmatrix}
a_1(x) & b_1(x)\\
c_1(x) & d_1(x)
\end{pmatrix}\]
satisfying that there exists $\mathcal{K} >0$ and $\eta >1$ such that
\[\begin{array}{ll}
\min \{ |a_1(x)|, |d_1(x)| \} > \mathcal{K},\\
 |d_1(x)| > \eta |a_1(x)|,\\
 \max \{ |b_1(x)|,|c_1(x)| \}  < \e.
\end{array}\]
Let $\delta < \frac{1}{1000}$, and $\e< \frac{1}{1000}$ satisfying lemma \ref{dmn}.
\vspace{3mm}

\textbf{Then, $f$ satisfies the following properties:}
\begin{enumerate}
\item[(1)] {$f$ has the dominated splitting $E \oplus F$;}
\item[(2)] {$\angle (E^s, E) < \delta, \angle (E^u, F) < \delta$;}
\item[(3)] { $f(x) = f_A (x), x \not\in B(p,\e)$;}
\item[(4)] {$f$ has two fixed points in $B(p,\e)$ : source $p$ and sink $q$;}
\item[(5)] {the length of two components of $W^{uu}_{ \frac{1}{10}}(p,f)-p$ both equals $\frac{1}{10}$, and \\
$2 \e ||Df|| < \frac{1}{10} $. }
\end{enumerate}
\end{Example}
A curve $\gamma^E$ is said to be an $E$-curve, if $\gamma^E$ is tangent to $E$ everywhere. Similarly, we define $F$-curve. The non-integrability of $E$ results from the following fact:

\begin{Lemma}Let $B(q)$ be the intersection of $B(p,\e)$ and the basin of the sink $q$.
For any $x \in B(q)$ and any $E$-curve $\gamma^E$ of length $3 \e$ centered at $x$, we have that $p  \in  \gamma^E$.
\end{Lemma}

\begin{proof}
We give the natural order on the curve $W^{uu}_{\frac{1}{10}}(p,f)$. By the dominated splitting, it is not difficult to show that,
$$\mathrm{T}{W^{uu}_{\frac{1}{10}}(p,f)}= F|_{W^{uu}_{\frac{1}{10}}(p,f)}.$$
Let $I$ be the set of points of the intersections of $W^{uu}_{\frac{1}{10}}(p,f)$ and any $E$-curve $\gamma^E$ of length $3 \e$ centered at any $x \in B(q)$. Note that,
$$q \in B(p,\e), \angle (E^s, E) < \delta, \angle (E^u, F) < \delta, {\rm \ and \ } \frac{1}{10} >> \e.$$
 Then, it is not difficult to deduce that every intersection above is exactly one point. Also, the lower bound $a$ and upper bound $b$ of $I$ satisfy that
$$\max \{d(a,p),d(b,p) \} < 2 \e.$$
 Suppose on the contrary that $I \neq \{p\}$, say $b\not=p$.
 Then, we can take a point $y \in I$ close enough to $b$. By the definition of $I$, we take an $E$-curve $\gamma^E$ starting from $y$ to $B(q)$ of length smaller than  $3 \e$(see the following picture).
 \begin{figure}[h!]
\includegraphics[height=5cm, width=10cm]{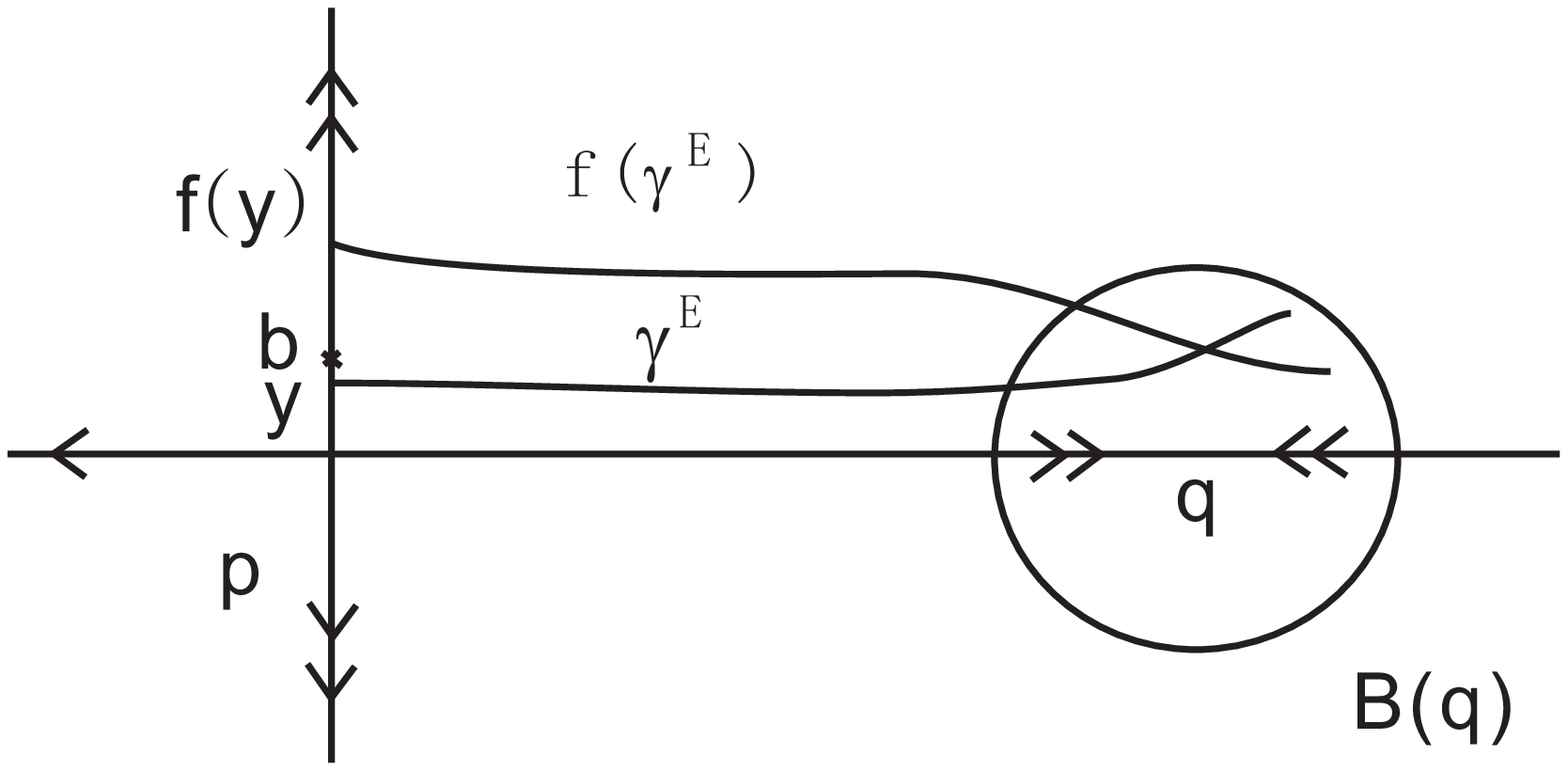}

\end{figure}
 \\Note that $f=f_A$ outside $B(p,\e)$, and $\angle (E^s, E) < \delta$. Then,
 $$\ell(f(\gamma^E)) < 3\e. $$
 By $2\e ||Df|| < \frac{1}{10} $, we see that
 $$f([a,b]) \subset W^{uu}_{\frac{1}{10}}(p,f). $$
 Then, $f(y) \in I$. By the uniform expanding of $f$ on the curve $W^{uu}_{\frac{1}{10}}(p,f)$, we see that the intersection $f(y) \not\in [a,b]$ . This contradiction finishes the proof of the lemma.
\end{proof}
{\noindent \bf Proof of the robust non-integrability of $E$.} Note that the above five properties of $f$ are robust. Then, by the above lemma, there exists an open neighborhood $\U$$ \ni f$ such that for any $g \in \U$, $g$ has the dominated splitting $E_g \oplus F_g$ , but $E_g$ is non-integrable.
\begin{Remark}Consider another saddle $q'$ in the $W^{s}(p,f_A)$. Then, the phenomenon in the above lemma also happens between saddle $q'$ and sink $q$.
\end{Remark}
\vspace{3mm}

{\noindent \bf Proof of Theorem \ref{thm}.} Let $p_1$ and $p_2$ be the two fixed point of $f_A$. Take a small enough $\e >0$. We construct the map $f$ as follows:
\begin{enumerate}
\item[(1)] {make the same perturbation in $B(p_1,\e)$ as the example above;}
\item[(2)] { make the symmetrical perturbation in $B(p_2,\e)$: for $f^{-1}_A$, we make the same perturbation in $B(p_2,\e)$ as the example above;}
\item[(3)] { $B(p_1,\e)$ and $B(p_2,\e)$ are disjoint. Also, $f$ has that
\begin{itemize}
\item {the length of two components of $W^{uu}_{ \frac{1}{10}}(p_1,f)-p_1$ both equals $\frac{1}{10}$,}
\item {the length of two components of $W^{ss}_{ \frac{1}{10}}(p_2,f)-p_2$ both equals $\frac{1}{10}$.}
\end{itemize}}
\end{enumerate}
Now, $f$ satisfies all properties in the theorem.

\end{document}